\documentclass[journal]{IEEEtran}
\usepackage{amsmath,amsfonts,amsthm}
\usepackage{array}
\usepackage[caption=false,font=normalsize,labelfont=sf,textfont=sf]{subfig}
\usepackage{textcomp}
\usepackage{stfloats}
\usepackage{url}
\usepackage{verbatim}
\usepackage{graphicx}
\usepackage{xcolor}
\usepackage{nomencl}
\usepackage{etoolbox}
\usepackage{ifthen}
\usepackage{booktabs}
\usepackage[ruled]{algorithm2e}

\usepackage[normalem]{ulem}

\usepackage{multirow}
\usepackage{bm}
\usepackage{makecell}
\usepackage{enumitem}

\allowdisplaybreaks 
\usepackage[subtle,tracking=normal]{savetrees}

\usepackage{cite}

\def\BibTeX{{\rm B\kern-.05em{\sc i\kern-.025em b}\kern-.08em
    T\kern-.1667em\lower.7ex\hbox{E}\kern-.125emX}}
\usepackage{balance}

\theoremstyle{definition}

\newtheorem{theorem}{Theorem}

\newtheorem{prop}{Proposition}

\usepackage{url}
\usepackage{hyperref} 
\usepackage{xcolor}

\hypersetup{
    colorlinks=true, 
    linkcolor=blue, 
    filecolor=black, 
    urlcolor=blue, 
    citecolor=blue, 
    linkbordercolor={1 1 1}, 
    pdfborder={0 0 0} 
}
\begin{document}

\bstctlcite{IEEEexample:BSTcontrol}
\title{Online Convex Optimization for Coordinated Long-Term and Short-Term Isolated Microgrid Dispatch}
%

\author{\IEEEauthorblockN{Ning Qi, \textit{Member, IEEE}}, \IEEEauthorblockN{Yousuf Baker, \textit{Graduate Student Member, IEEE}}, \IEEEauthorblockN{Bolun Xu, \textit{Member, IEEE}}

\thanks{The work was partly supported by the Department of Energy, Office of Electricity, Advanced Grid Modeling Program under contract DE-AC02-05CH11231 and partly supported by the National Science Foundation under award ECCS-2239046. Yousuf Baker's work was supported by the U.A.E. Ministry of Education.

Ning Qi, Yousuf Baker, and Bolun Xu are with the Department of Earth and Environmental Engineering, Columbia University, New York, NY
10027 USA (e-mail: \{nq2176, ykb2105, bx2177\}@columbia.edu).}
}

\markboth{IEEE TRANSACTIONS ON Smart Grid,~Vol.~X, No.~X, XX July~2025}
{How to Use the IEEEtran \LaTeX \ Templates}

\maketitle


\begin{abstract}

This paper proposes a novel non-anticipatory coordinated dispatch framework for an isolated microgrid with hybrid short- and long-duration energy storage (LDES). We introduce a convex hull approximation model for nonconvex LDES electrochemical dynamics, facilitating computational tractability and accuracy. To address temporal coupling in SoC dynamics and long-term contracts, we generate hindsight-optimal state-of-charge (SoC) trajectories of LDES and netloads for offline training. In the online stage, we employ kernel regression to dynamically update the SoC reference and propose an adaptive online convex optimization (OCO) algorithm with SoC reference tracking and expert tracking to mitigate myopia and enable adaptive step-size optimization. We rigorously prove that both long-term and short-term policies achieve sublinear regret bounds over time, which improves with more regression scenarios, stronger tracking penalties, and finer convex approximations.  Simulation results show that the proposed method outperforms state-of-the-art methods, reducing costs by 73.4\%, eliminating load loss via reference tracking, and achieving an additional 2.4\% cost saving via the OCO algorithm. These benefits scale up with longer LDES durations, and the method demonstrates resilience to poor forecasts and unexpected system faults.  
\end{abstract}

\section{Introduction}\label{Introduction}

Ensuring a reliable power supply for isolated microgrids in remote and underserved regions is challenging due to limited or nonexistent access to the main power grid~\cite{isolated}. These microgrids primarily rely on intermittent and stochastic renewable energy sources (RES), supported by a limited set of flexibility resources: diesel generators (DG), battery energy storage (BES), and long-duration energy storage (LDES) like pumped hydro and hydrogen~\cite{wang2025feasibility}. These resources offer complementary capabilities: DG and BES provide fast and reliable power responses but are constrained by limited energy capacity, while LDES systems can store large amounts of energy but are often limited in their short-term power output and relatively low conversion efficiency. As a result, microgrid operators must carefully coordinate this diverse mix of resources, accounting for both short-term (daily) and long-term (seasonal) variations in demand and renewable generations.


The key challenge in microgrid dispatch with hybrid short- and long-duration storage resources is the coordination of dispatch policies across multiple timescales under uncertainty. On one hand, long-term dispatch faces difficulties in accurately capturing seasonal uncertainty patterns and establishing effective boundaries or guidance for short-term dispatch~\cite{zhang2025long}, particularly given the absence of reliable long-term forecasts. On the other hand, short-term dispatch encounters challenges in capturing future opportunities and risks~\cite{baker2023transferable} without relying on future information, thus requiring a non-anticipatory approach that dynamically adapts to real-time uncertainties under the guidance of long-term policy. Furthermore, nonconvex characteristics in LDES models, such as electrochemical dynamics in electrolyzer~\cite{sanchez2018semi}, significantly increase computational complexity. Hence, it is crucial for the microgrid operator to develop a tractable dispatch framework to coordinate long-term and short-term operations, incorporating seasonal variability and immediate uncertainty realizations into decision making.

This paper proposes a novel non-anticipatory long-short-term coordinated dispatch framework for microgrids with hybrid short-long-duration storage resources. The proposed framework provides microgrid operators with a near-long-term optimal, prediction-free, resilient, and computationally tractable online dispatch tool. Our contributions are as follows:

\begin{enumerate}
\item \textit{Convex LDES Model:} We develop a convex hull approximation model for LDES to address the nonconvex electrochemical dynamics in electrolyzer and fuel cell, which is more accurate and scalable than piecewise linear models. The convex hull model avoids binary variables and preserves compatibility with gradient-based optimization.
\item \textit{Non-Anticipatory Dispatch:} We propose a long-short-term coordinated dispatch framework. We generate hindsight-optimal SoC trajectories and netloads for offline training. In the online stage, we employ kernel regression to dynamically update SoC reference and propose an adaptive online convex optimization (OCO) algorithm with reference tracking and expert tracking to mitigate myopia and enable adaptive step-size optimization.
\item \textit{Theoretical Guarantees:} We prove that both long-term and short-term policies achieve sublinear regret bounds, improving with more regression scenarios, stronger tracking penalties, and finer convex approximations.
\item \textit{Simulation Analysis:} 
We validate the proposed method on an isolated microgrid in Alaska. The proposed method outperforms state-of-the-art methods by reducing costs by 73.4\% and eliminating load loss through reference tracking, with an additional cost reduction of 2.4\% achieved via the OCO algorithm. These benefits scale with longer LDES durations, and the method demonstrates resilience to poor forecasts and unexpected system faults.
\end{enumerate}


We organize the remainder of this paper as follows. Section~\ref{LR} summarizes the previous works on microgrid dispatch and online optimization. Section~\ref{PR}~provides problem formulation and preliminaries for microgrid dispatch. Section~\ref{framework} proposes the non-anticipatory dispatch framework and theoretical analysis. Section~\ref{Case Study}~describes case studies to verify the theoretical results. 
Finally, section~\ref{Conclusion}~concludes this paper.

\section{Literature Review}\label{LR}

Existing studies primarily focus on short-term microgrid dispatch with hybrid short and long-duration storage resources. To address uncertainties in microgrid, early studies primarily rely on two-stage optimization methods, including robust optimization~\cite{hashemifar2022two,gu2023robust}, stochastic optimization~\cite{darivianakis2017data,dong2023forecast}, distributionally robust optimization~\cite{zhou2024distributionally}. These methods assume real-time decisions are made after all uncertainties are observed, while microgrid dispatch must respond to uncertainty information available up to the current time, without knowledge of future realizations. This motivates the exploration of online optimization methods for real-time dispatch by leveraging continuously updated uncertainty realizations. 

Microgrid operators require methods that are reliable, interpretable, fast-responding to real-time uncertainties, and resilient to extreme events. Many methods such as model predictive control (MPC)~\cite{guo2023long,abdelghany2024hierarchical}, reinforcement learning (RL)~\cite{hu2022soft,zhang2025long}, and stochastic dynamic programming (SDP)~\cite{schaffer2022stochastic,li2021multi} are promising, but they have certain limitations. MPC methods are limited by the poor forecast accuracy in microgrids, RL methods face challenges in transparency and robustness, SDP methods suffer from high computational complexity and limited scalability, making them less practical for real-time microgrid dispatch.

In addition to the aforementioned methods, there is growing interest in applying online control algorithms to power systems, such as Lyapunov optimization, online feedback optimization, and OCO, as they provide rapid responses to real-time uncertainties and theoretical guarantees. Lyapunov optimization~\cite{shi2015real,yu2022joint} adopts a ``1-lookahead" decision pattern, where decisions are made sequentially by observing current uncertainties and minimizing a drift-plus-penalty term derived from the Lyapunov function to balance immediate costs and queue stability. However, the ``observe-then-act'' pattern is unsuitable for practical microgrid dispatch, since decisions typically need to be made before uncertainties are realized. In contrast, OCO~\cite{kim2019predicting,lesage2019online} adopts a "0-lookahead" decision pattern, employing projection-based methods to proactively make decisions based on past information, and subsequently updating decisions based on observed regrets. Online feedback optimization~\cite{bernstein2019real,yan2023real} operates similarly to OCO by leveraging real-time measurement feedback and gradient-based methods, offering simplicity and rapid convergence, but it lacks theoretical regret guarantees compared to OCO. Although these methods have been applied in power systems, their application in LDES remains limited, primarily due to their myopic nature and inability to handle time-coupling constraints, such as SoC dynamics and long-term contracts in LDES~\cite{elizahydropower}. Directly applying these methods for long-term dispatch over monthly or even annual horizons can result in suboptimality and prolonged power shortages, as demonstrated in recent studies~\cite{qi2025long,guo2023long,zhang2025long}.

Recent studies have recognized the importance of long-term dispatch and its coordination with short-term dispatch. For instance, the annual reservoir strategy is generated using a Markov Decision Process and is integrated with the intraday hourly dispatch through deep reinforcement learning~\cite{zhang2025long}. The long-term boundaries for hydrogen storage SoC are established in~\cite{wakui2022shrinking} based on historical uncertainty scenarios. Hydrogen storage SoC references are generated and recursively updated online through kernel regression~\cite{guo2023long,qi2025long,kang2025securing}. Although these studies address seasonal uncertainty patterns and provide long-term dispatch policies to effectively guide short-term decisions, the long-term policies remain somewhat heuristic and lack rigorous theoretical performance guarantees. Moreover, most existing studies assume a linear model with constant efficiency~\cite{guo2023long}, adopt piecewise linear approximations~\cite{qi2025long}, or rely on learning-based methods~\cite{zhang2025long} to handle the inherent nonconvexity of LDES. However, these methods struggle to balance accuracy, tractability, and the preservation of convexity. To this end, the dispatch framework proposed in this work effectively coordinates long-term and short-term operations, while tractably addressing nonconvexity in LDES. It extends the OCO method to explicitly handle time-coupling constraints through long-term reference learning and tracking, providing rigorous theoretical guarantees for both long-term and short-term policies.

\section{Problem Formulation and Preliminaries}\label{PR}

\subsection{Oracle Multi-Period Economic Dispatch}

We consider an isolated microgrid that integrates DG, RES, BES, LDES and local load. Specifically, RES consists of wind and solar, while LDES comprises water electrolyzer, hydrogen storage, and fuel cell~\cite{qi2025long}. The \textbf{oracle multi-period economic dispatch (OED)} is formulated in~\eqref{eq:OED}.
\begin{subequations}\label{eq:OED}
\begin{align}
&\min \sum_{t\in\mathcal{T}}\big(c^{\text{l}}l_{t}+c^{\text{d}}d_{t}+c^{\text{b}}b_{t}^{+}+c^{\text{h}} p_{t}^{+}\big)\label{objective} \\
\text{s.t. }&\underline{D}\leq d_{t}\leq\overline{D}\text{, }\forall t\in \mathcal{T}\label{DGpower} \\
&0\le b_{t}^{-}\le \overline{B}\text{, } 0\le b_{t}^{+}\le \overline{B}\text{, }\forall t\in \mathcal{T}\label{power-bound-B}\\
&e_{t+1}=e_{t}+\eta^{-}b_{t}^{-}-b_{t}^{+}/\eta^{+}\text{, }\forall t\in \mathcal{T}\label{E-power-B}\\ 
&\underline{E}\le e_{t}\le \overline{E}\text{, }\forall t\in \mathcal{T}\label{E-bound-B}\\
&0\leq l_{t}\leq L_{t}\text{, }\forall t\in \mathcal{T}\label{Loadcurtailment}\\
&0\leq r_{t}\leq R_{t}\text{, }\forall t\in \mathcal{T}\label{REScurtailment}\\
&r_{t}+d_{t}+b_{t}^{+}-b_{t}^{-}+p_{t}^{+}-p_{t}^{-}+l_{t}= L_{t}\text{, }\forall t\in \mathcal{T}\label{powerbalance}\\
&\underline{P}\le p_{t}^{-}\le \overline{P}\text{, } 0\le p_{t}^{+}\le \overline{P}\text{, }\forall t\in \mathcal{T}\label{power-bound-H}\\
&\underline{H}\le h_{t}\le \overline{H}\text{, }\forall t\in \mathcal{T}\label{E-bound-H}\\
&h_{T}\geq {H}_{T}\text{, }\forall t\in \mathcal{T}\label{E-balance-H}\\
&h_{t}=h_{t-1}+\zeta^{-}(p_{t}^{-})p_{t}^{-}-p_{t}^{+}/\zeta^{+}(p_{t}^{+})\text{, }\forall t\in \mathcal{T}\label{E-power-H} 
\end{align}
\end{subequations}
where $\mathcal{T}$ denotes the set of time periods. Decision variables include the scheduled renewable generation $r_{t}$, load shedding power $l_{t}$, DG output $d_{t}$, BES discharge power $b_{t}^{+}$, BES charge power $b_{t}^{-}$, and BES SoC $e_{t}$, as well as LDES discharge power $p_{t}^{+}$, LDES charge power $p_{t}^{-}$, and LDES SoC $h_{t}$. Parameters are defined as follows: $c^{\text{l}}$, $c^{\text{d}}$, $c^{\text{b}}$, and $c^{\text{h}}$ denote the penalty cost of load shedding, fuel price of DG, marginal degradation cost of BES and LDES. $\underline{D}$ and $\overline{D}$ denote the lower and upper power bounds of DG. $\overline{B}$, $\underline{E}$ and $\overline{E}$ denote the upper power bound, lower and upper SoC bounds of BES. $\eta^{-}$ and $\eta^{+}$ denote charge and discharge efficiency of BES. $\underline{P}$, $\overline{P}$, $\underline{H}$ and $\overline{H}$ denote the lower and upper power bounds, lower and upper SoC bounds of LDES. $\zeta^{-}$ and $\zeta^{+}$ denote LDES charge and discharge efficiencies, represented as nonconvex functions of the charge and discharge power~\cite{qi2025long}. ${H}_{T}$ denotes the final SoC target required by the long-term contract. $L_{t}$ and $R_{t}$ denote the load power and available renewable power.

The objective function~\eqref{objective} aims to minimize the annual operational cost of the microgrid. Constraints~\eqref{DGpower} limit the power output of DG. Constraints~\eqref{power-bound-B} limit the charge and discharge power of BES. Constraints~\eqref{E-power-B} define the SoC dynamics with charge/discharge energy of BES. Constraints~\eqref{E-bound-B} limit the SoC of BES. Constraints~\eqref{Loadcurtailment}-\eqref{REScurtailment} limit the load shedding and scheduled renewable power. Constraints~\eqref{powerbalance} guarantee the power balance. Constraints~\eqref{power-bound-H} limit the charge and discharge power of LDES. Constraints~\eqref{E-power-H} define the SoC dynamics with charge/discharge energy of LDES. Constraints~\eqref{E-bound-H} limit the SoC of LDES. Constraint~\eqref{E-balance-H} ensures LDES SoC recycling and compliance with long-term contracts, while it is not necessary for BES. The complementary constraints that prevent simultaneous charge and discharge of BES and LDES are relaxed and removed from the model since the sufficient condition is guaranteed~\cite{li2015sufficient}. OED is computationally intractable and practically unattainable due to: (1) inherent nonconvexity in the LDES model, and (2) the non-anticipatory constraints~\eqref{Loadcurtailment}-\eqref{powerbalance}. Nevertheless, it provides a theoretical benchmark for subsequent analyses.

\subsection{Convex Reformulation}
We propose an inner convex hull approximation model in~\eqref{eq:reformulation} to address nonconvex electrochemical dynamics in electrolyzer and fuel cell. As illustrated in Fig.~\ref{fig:convex_hull}, we first select samples (marked by red crosses) from the nonconvex power-hydrogen curve, serving as vertices of the convex hull. The approximated decision (blue star) is then obtained via a convex combination of these vertices. We note that increasing the number of samples (i.e., finer convex approximations) reduces approximation error at the expense of higher computational complexity, as theoretically proven in Appendix~\ref{appendix2}. Traditional piecewise linear approximations~\cite{qi2025long,raheli2023conic} can also address this nonconvexity by introducing binary variables to enforce segment selection, thereby maintaining the nonconvexity. In contrast, our convex hull approximation ensures convexity and significantly improves tractability.
\begin{subequations}\label{eq:reformulation}
\begin{align}
&\min \sum_{t\in\mathcal{T}}\big(c^{\text{l}}l_{t}+c^{\text{d}}d_{t}+c^{\text{b}}b_{t}^{+}+c^{\text{h}} \sum_{m\in\mathcal{M}}\lambda_{m\text{,}t}^{+}P_{m}^{+}\big)\label{objective1} \\
\hspace{-0.2cm}\text{s.t. }&\forall t\in \mathcal{T}\text{, }\eqref{DGpower}-\eqref{REScurtailment}\text{, }\eqref{E-bound-H}-\eqref{E-balance-H}\notag\\
&r_{t}+d_{t}+b_{t}^{+}-b_{t}^{-}+\sum_{m\in\mathcal{M}}(\lambda_{m\text{,}t}^{+}P_{m}^{+}-\lambda_{m\text{,}t}^{-}P_{m}^{-})+l_{t}= L_{t}\label{powerbalance-1}\\
&h_{t}=h_{t-1}+\sum\nolimits_{m\in\mathcal{M}}(\lambda_{m\text{,}t}^{-}H_{m}^{-}-\lambda_{m\text{,}t}^{+}H_{m}^{+})\label{E-power-H-1}\\
& \lambda_{m\text{,}t}^{-}\text{,}\lambda_{m\text{,}t}^{+}\geq 0\text{, }\sum\nolimits_{m\in\mathcal{M}}\lambda_{m\text{,}t}^{-} = 1\text{, }\sum\nolimits_{m\in\mathcal{M}}\lambda_{m\text{,}t}^{+} = 1\label{piecewise}
\end{align}
\end{subequations}
where $\mathcal{M}$ denote the set of discretized samples. $P_{m}^{-}$, $H_{m}^{-}$ and $P_{m}^{+}$, $H_{m}^{+}$ represent sampled LDES charge/discharge power and hydrogen production/consumption. $\lambda_{m,t}^{-}$, $\lambda_{m,t}^{+}$ denote combination decision variables for LDES charge/discharge power.
\begin{figure}[t!] 
    \setlength{\abovecaptionskip}{-0.1cm}  
    \setlength{\belowcaptionskip}{-0.1cm}   
\centerline{\includegraphics[width=0.9\columnwidth]{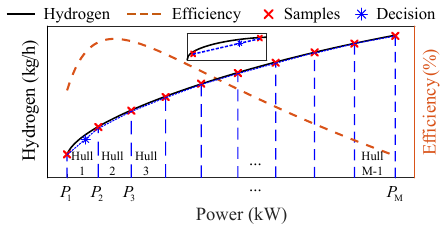}}
    \caption{Illustration of Convex Hull Approximation for the LDES Model.}    \label{fig:convex_hull}
\end{figure}


\subsection{Single-Period Economic Dispatch}

Real-time dispatch of microgrid considers either a single time period or a short look-ahead window. We consider the \textbf{single-period economic dispatch (SED)} as formulated in~\eqref{eq:SED}. To mitigate the inherent myopia of SED, one promising approach is to incorporate the future opportunity cost of BES and LDES into the objective function using SDP~\cite{qi2025locational}. However, learning the opportunity value functions of BES and LDES is intractable due to their strong coupling and the exponential state complexity. To address this challenge, we propose an alternative method that directly learns the SoC trajectories of BES and LDES. In Proposition 1, we show that determining SoC trajectories is equivalent to addressing the time-coupling constraints, thus effectively mitigating the myopia of SED.
\begin{align}\label{eq:SED}
&\min\big(c^{\text{l}}l_{t}+c^{\text{d}}d_{t}+c^{\text{b}}b_{t}^{+}+c^{\text{h}} \sum\nolimits_{m\in\mathcal{M}}\lambda_{m\text{,}t}^{+}P_{m}^{+}\big) \\
\text{s.t. }&\eqref{DGpower}-\eqref{REScurtailment}\text{, }\eqref{E-bound-H}\text{, }\eqref{powerbalance-1}-\eqref{piecewise}\notag
\end{align}

\begin{prop}\label{prop:separation}\textbf{Solution Equivalence.}
Let \( \{x_t^\star\text{,} y_t^\star\}_{t=1}^T \) be the optimal solution to problem~\eqref{eq:reformulation}, where $x_t=\{r_{t}\text{,}l_{t}\text{,}d_{t}\}$ are subject to single-period constraints, $y_t=\{b_{t}^{+}\text{,}b_{t}^{-}\text{,}e_{t}\text{,}\lambda_{m\text{,}t}^{+}\text{,}\lambda_{m\text{,}t}^{-}\text{,}h_{t}\}$ are subject to inter-temporal constraints. If SED with $y_t = y_t^\star$ and is strictly convex in $x_t$ for each $t$, then solving SED sequentially yields the same optimal solution $x_t^\star$ to problem~\eqref{eq:reformulation}.
\end{prop}
\begin{proof}
Given $y_t=y_t^\star$, problem~\eqref{eq:reformulation} can be decomposed into \( T \) independent subproblems, each of the form \eqref{eq:SED}. By convexity and uniqueness of each subproblem, solving each \eqref{eq:SED} yields the same \( x_t^\star \) as the optimal solution to problem~\eqref{eq:reformulation}.  
\end{proof}

\begin{figure*}[t] 
    \setlength{\abovecaptionskip}{-0.1cm}  
    \setlength{\belowcaptionskip}{-0.1cm}   
    \centerline{\includegraphics[width=2\columnwidth]{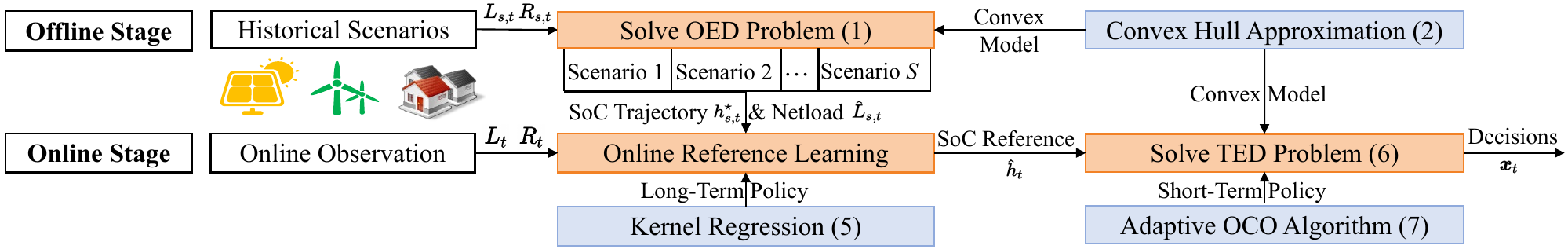}}
    \caption{Non-anticipatory long-short-term coordinated dispatch framework.}
    \captionsetup{justification=centering}
    \label{fig:framework}
\end{figure*}

\section{Dispatch Framework and Regret Analysis}\label{framework}

We propose a non-anticipatory long-short-term coordinated dispatch framework (Fig.\ref{fig:framework}) combining offline training and online control. In the offline stage, we generate long-term scenarios of renewable generation and load using historical or synthetic data\cite{li2024long,fu2025knowledge} to enhance robustness against extreme events. The online model is trained using the hindsight-optimal SoC solved from the offline problem, along with netload data. In the online stage, SoC references are updated in real-time via kernel regression, while an adaptive OCO algorithm with reference and expert tracking ensures near-long-term optimal, non-anticipatory decisions. We further establish sublinear regret bounds to provide theoretical performance guarantees.



\subsection{Online SoC Reference Learning}

Inspired by Proposition~\ref{prop:separation}, we propose a non-parametric conditional imitation policy for online SoC reference learning based on kernel regression, whose estimation accuracy can be rigorously characterized by the mean squared error (MSE) in~\eqref{eq:mse_kernel}. The MSE is decomposed into bias and variance, which provides critical theoretical insights for our algorithm design.
\begin{equation}\label{eq:mse_kernel}
\text{MSE}=\underbrace{\frac{\sigma^4}{4}\left(\int v^2K(v)dv\right)^2\Lambda^2}_{\text{Bias}^2}+\underbrace{\frac{\nu^2}{S\sigma^\iota\pi(\xi)}\int K^2(v)dv}_{\text{Variance}}
\end{equation}
\noindent where $\sigma$ controls the smoothness and locality of the kernel estimator. $\iota$ denotes the dimension of the input vectors. $\Lambda$ denotes the Lipschitz constant and represents the smoothness of the underlying mapping. $\nu^2$ denotes the conditional variance. $\pi(\xi)$ denotes the density of input data $\xi$. $K$ is the kernel function. $S$ is the sample size (input scenario number).

The proposed learning algorithm involves procedures for data initialization, offline training for optimal selection of bandwidth $\sigma$ and window size $W$, and online learning, thereby systematically and effectively reducing estimation errors.

\textbf{Step 1: Data Initialization.} We construct the input vector of normalized netload and historical SoC from periods $t-W$ to $t-1$. The training inputs ${\xi}_{s\text{,}t}$ and test inputs ${\xi}_{t}$ are defined in~\eqref{historicalsequence}-\eqref{observesequence}. Compared to previous work~\cite{guo2023long,qi2025long}, we introduce several theoretical improvements: (1) we reduce input dimensionality by using netload $\hat{L}_t$ instead of separate load and RES data, and adopt a rolling optimal window size rather than expanding it with time; (2) normalization improves numerical stability and bandwidth selection; (3) incorporating historical SoC explicitly captures temporal dependencies, reducing variance and improving prediction; and (4) we focus learning on LDES SoC due to its more stable long-term dynamics, while treating the rapidly fluctuating BES SoC as a slack variable in online optimization.

\textbf{Step 2: Offline Training.} The only hyperparameters requiring tuning are the bandwidth $\sigma$ and the window size $W$. Specifically, increasing the window size reduces bias by capturing richer temporal information but simultaneously enlarges the input dimension, thereby increasing variance. On the other hand, the bandwidth controls the degree of kernel smoothing: a larger bandwidth reduces variance at the cost of increased bias. In practical implementation, we first select discrete candidate values for the window size $W$. For each candidate $W$, the optimal bandwidth $\sigma^*$ is determined by minimizing the MSE as in~\eqref{sigmaoptimal}. Subsequently, the optimal window size $W^*$ is selected as the one achieving the minimal MSE. Additionally, the scenario number $S$ and kernel type $K$ also influence the learning performance. These factors will be analyzed in detail through the subsequent case studies.

\textbf{Step 3: Online Learning.} The similarity between ${\xi}_{t}$ observed online and each historical scenario ${\xi}_{s,t}$ is measured using the Gaussian kernel function based on the Euclidean distance, yielding dynamic weights $\omega_{s,t}$ as defined in~\eqref{weight}. To capture cumulative temporal effects, the Gaussian kernel function is modified by the window size. Finally, the dynamic weights are updated, and the corresponding SoC trajectory $\hat{h}_{t}$ for LDES is computed as the weighted average value in~\eqref{referweight}.
\begin{subequations}\label{refupdate}
\begin{align}
& {\xi}_{s\text{,}t}=\{\hat{L}_{s\text{,}t-W}\text{,}h^{\star}_{s\text{,}t-W}\text{,}\cdots\text{,}\hat{L}_{s\text{,}t-1}\text{,}h^{\star}_{s\text{,}t-1}\}\label{historicalsequence}\\
&{\xi}_{t}=\{\hat{L}_{t-W}\text{,}h^{}_{t-W}\text{,}\cdots\text{,}\hat{L}_{t-1}\text{,}h_{t-1}\}
   \label{observesequence} \\
   & \sigma^\star(W^\star) = \left(\frac{2W^\star \nu^2 \int K^2(v) dv}{N \Lambda^2 (\int v^2 K(v) dv)^2 \pi(\xi)}\right)^{\frac{1}{2W^\star+4}}\label{sigmaoptimal}\\
   &\omega_{s\text{,}t}=\dfrac{K_{t}(\xi_{t}\text{,}\xi_{s\text{,}t})}{\sum_{{s^{\prime}=1}}^{S}K_{t}(\xi_{t}\text{,}\xi_{s^{\prime}\text{,}{t}})}\text{,}\ K_{t}(\xi\text{,}\xi^{'})=e^{{-\frac{(\|\xi-\xi^{'}\|_{2})^{2}}{W^\star{\sigma^\star}^2}}}\label{weight}\\
&\hat{h}_{t}=\sum\nolimits_{{s=1}}^{S}\omega_{s\text{,}t} h_{s\text{,}t}^{\star}\label{referweight}
\end{align}
\end{subequations}

Compared to neural network-based learning approaches, the proposed method offers several notable advantages. First, it features low computational complexity, enabling fast real-time updates suitable for online optimization. Second, it requires minimal hyperparameter tuning, thus simplifying model deployment. Additionally, it naturally adapts to system states and configuration changes, as it only requires updating offline-generated scenarios without modifying model structures. Finally, its inherent interpretability facilitates easier analysis and troubleshooting, enhancing the reliable microgrid operations.

\subsection{Online Dispatch Optimization}

To incorporate opportunities into SED, we propose a \textbf{trajectory-guided single-period economic dispatch (TED)}, which includes a quadratic penalty term $\|h_t-\hat{h}_t\|^2$ with coefficient $\theta$ to track the long-term SoC reference. We prove in the Appendix~\ref{appendix1} that tracking the long-term reference is equivalent to explicitly addressing the time-coupling constraints~\eqref{E-balance-H} and~\eqref{E-power-H-1}. \eqref{eq:TED} shows the TED formulation. 
\begin{align}\label{eq:TED}
\hspace{-0.3cm}&\min\big(c^{\text{l}}l_{t}+c^{\text{d}}d_{t}+c^{\text{b}}b_{t}^{+}+c^{\text{h}} \sum_{m\in\mathcal{M}}\lambda_{m\text{,}t}^{+}P_{m}^{+}\big)+\theta\|h_t-\hat{h}_t\|^2 \\
\hspace{-0.3cm}\text{s.t. }&\eqref{DGpower}-\eqref{REScurtailment}\text{, }\eqref{E-bound-H}\text{, }\eqref{powerbalance-1}-\eqref{piecewise}\notag
\end{align}

TED admits a compact form in~\eqref{eq:compact}. $f_t$, $g_t$, and $\boldsymbol{x}_{t}$ denote the time-varying objective function, constraints, and decision variables, respectively. We propose an adaptive virtual-queue-based OCO algorithm to achieve a non-anticipatory decision policy that does not rely on long-term forecasts. We incorporate expert-tracking to address the inherent sensitivity of OCO to the stepsize, where each expert $i$ operates with distinct stepsizes $\alpha_{i\text{,}t-1}$ \text{and} $\beta_{i\text{,}t-1}$. At each time step, the virtual queue $Q_{i\text{,}t}$ and decisions $\boldsymbol{x}_{i\text{,}t}$ are updated in~\eqref{eq:Q_update} and~\eqref{eq:x_update} for each expert in parallel, respectively. $\left\langle \psi\text{,}\psi'\right\rangle$ denotes the standard inner product of two vector $\psi$ and $\psi'$. $X$ denotes the feasible sets. Subsequently, the surrogate loss $\ell{i,t}$ is computed via \eqref{eq:loss}. Finally, the weights $\rho_{i,t}$ assigned to each expert are updated based on their empirical performance measured by the surrogate losses, as detailed in \eqref{eq:final_decision}. The ultimate decision is determined as the weighted average of all experts' decisions.
\begin{subequations}\label{OCO}
\begin{align}
& \min \ f_t(\boldsymbol{x}_{t})\quad \text{s.t. } g_t(\boldsymbol{x}_{t})\leq0\text{, }\boldsymbol{x}_{t}=\{x_t\text{,}y_t\}  \label{eq:compact}\\
&Q_{i\text{,}t-1}=Q_{i\text{,}t-2}+\beta_{i\text{,}t-1}[{g}_{t-1}(\boldsymbol{x}_{i\text{,}t-1})]_+\label{eq:Q_update}\\
& \boldsymbol{x}_{i\text{,}t}=\arg\min_{\boldsymbol{x}\in \mathcal{X}}\{\alpha_{i\text{,}t-1}\left\langle\partial f_{t-1}(\boldsymbol{x}_{i\text{,}t-1})\text{,}\ \boldsymbol{x}-\boldsymbol{x}_{i\text{,}t-1}\right\rangle\label{eq:x_update}\\
&\hspace{0.5cm}+\alpha_{i\text{,}t-1}\beta_{t-1}\left\langle Q_{i\text{,}t-1}\text{,}\ [g_{t-1}(\boldsymbol{x})]_+\right\rangle+\|\boldsymbol{x}-\boldsymbol{x}_{i\text{,}t-1}\|^2\}\notag \\
& \ell_{i\text{,}t-1}=\langle\partial {f}_{t-1}(\boldsymbol{x}_{t-1})\text{,}\ \boldsymbol{x}_{i\text{,}t-1}-\boldsymbol{x}_{t-1}\rangle\label{eq:loss}\\
& \rho_{i\text{,}t}=\frac{\rho_{i\text{,}t-1}e^{-\gamma \ell_{i\text{,}t-1}}}{\sum_{i=1}^N\rho_{i\text{,}t-1}e^{-\gamma \ell_{i\text{,}t-1}}}\text{, }\boldsymbol{x}_t=\sum\nolimits_{i=1}^N\rho_{i\text{,}t}\boldsymbol{x}_{i\text{,}t}\label{eq:final_decision}
\end{align}
\end{subequations}

The key idea of the proposed OCO algorithm is to leverage information from the previous time step to approximate the current system state, while maintaining feasibility through adaptive virtual queues, penalties for constraint violations. Specially, ${f}_{t}(\boldsymbol{x}_t)$ and ${g}_t(\boldsymbol{x}_t)$ are approximated using the Taylor expansion $\left\langle\partial f_{t-1}(\boldsymbol{x}_{i\text{,}t-1})\text{,}\ \boldsymbol{x}-\boldsymbol{x}_{i\text{,}t-1}\right\rangle$ and clipped constraint function $\left[{g}_{t-1}(\boldsymbol{x})\right]_{+}$. Compared to the existing OCO framework, we utilize partial Lagrangian relaxation only for non-anticipatory constraints and generate virtual queues to substitute the corresponding dual variables. The value of the virtual queue increases when constraints are violated, dynamically adjusting the penalty associated with constraint violations. The Bregman divergence $\left|\boldsymbol{x}-\boldsymbol{x}_{i\text{,}t-1}\right|^2$ is introduced to ensure decision stability and convergence. Moreover, the weights $\rho_{i\text{,}t}$ prioritize experts demonstrating superior performance based on surrogate losses, enabling adaptive stepsize selection.

\subsection{Theoretical Analysis on Dynamic Regret Bound}
\begin{theorem}\label{tm}  \textbf{Sublinear Dynamic Regret Bound.}
Given convex functions $f_t$ and $g_t$ defined on a convex, closed set $\mathcal{X}$ with bounded diameter, assume $F\text{,} J > 0$ exist such that $\forall x\text{,}y \in \mathcal{X}$:
\begin{align}\label{condition}
|f_t(x)-f_t(y)|\text{,}\; \|g_t(x)\|\leq F\text{,}\quad\|\partial f_t(x)\|\text{,}\;\|\partial g_t(x)\|\leq J
\end{align}
With parameters set as in~\eqref{parameter}, we can achieve the dynamic regret bound in~\eqref{regret}.
\begin{equation}\label{parameter}
\begin{aligned}
& \alpha_{i\text{,}t}=\dfrac{\alpha_{0}2^{i-1}}{t^{c}}\text{,}\ \beta_{i\text{,}t}=\dfrac{\beta_{0}}{\sqrt{\alpha_{i\text{,}t}}}\text{,} \ \gamma=\dfrac{\gamma_{0}}{T^{c}}\text{, }M=M_0T^{c}\text{, }\\&\theta=\theta_{0}{T^{c}}\text{, }S=S_{0}{T^{c}}\text{, }N=\lfloor\kappa \log_{2}(1+T)\rfloor+1 \ 
\end{aligned}
\end{equation}
\begin{equation}\label{regret}
\begin{aligned}
\text{Regret}&=\mathcal{O}(T^c (1 + P_x)^{1 - \kappa} + T^{1 - c} (1 + P_x)^{\kappa})\\&+\mathcal{O}({T}^{{1-c/2}})+\mathcal{O}(T^{1-c})+\mathcal{O}(T^{1-2c})
\end{aligned} 
\end{equation}
where $\kappa \in [0\text{,}c]$, $c \in (0\text{,}1)$, $\alpha_0\text{, }\beta_0\text{, }M_0\text{, }\theta_0\text{, }S_0> 0$, $\gamma_0 \in (0\text{,} {1}/{\sqrt{2J}})$\text{, }$P_x=\sum_{t=1}^{T-1}\left\|\boldsymbol{x}_{t+1}-\boldsymbol{x}_t\right\|$.
\end{theorem}

Theorem~\ref{tm} demonstrates a sublinear dynamic regret bound for the proposed dispatch framework, which improves with more kernel regression scenarios, larger reference tracking penalty, and more convex hull segments. Compared to our previous work~\cite{qi2025long} and other state-of-the-art OCO algorithms, we focus on the performance of both long-term and short-term policies, rather than focusing solely on the short-term policy derived from the OCO algorithm itself. The derived dynamic regret bound provides the microgrid operator with a robust understanding of algorithm performance and facilitates better parameter tuning. The proposed dispatch framework is outlined in \textbf{Algorithm}~\ref{algorithm1}.  We defer the complete proof to Appendix~\ref{appendix1}.

\begin{algorithm}[t]\label{algorithm1}
\caption{Long-Short-Term Coordinated Dispatch}
\SetAlgoLined
\SetEndCharOfAlgoLine{}
\textbf{\hspace{-0.4cm}Stage1: Offline scenario and SoC trajectory generation}

\hspace{-0.2cm}\KwIn{Historical load ${L}_{s\text{,}t}$ and RES ${R}_{s\text{,}t}$.}
\hspace{-0.2cm}\KwOut{hindsight-optimal SoC trajectory of LDES ${h}_{s\text{,}t}^{\star}$.}
\SetKw{Parallel}{parallel}
    \For{$s=1$ \KwTo $S$ }{
        Solve OED with reformulation~\eqref{eq:reformulation} for each scenario.\;
    }
\textbf{\hspace{-0.4cm}Stage2: Online SoC reference learning and dispatch}

\hspace{-0.2cm}\KwIn{Parameters setting in~\eqref{parameter}. }
\hspace{-0.2cm}\KwOut{SoC reference $\hat{h}_{t}$ and dispatch decisions $x_{t}\text{, }y_{t}$.}
\SetKwBlock{StepOne}{Step 1 -Initialization}{}
\SetKwBlock{StepTwo}{Step 2 - Reference Update and Online Optimization}{}
\StepOne{
Set $Q_{i\text{,}0}=0\text{,}\ \boldsymbol{x}_{i\text{,}1}\in \mathcal{X}\text{,}\ \boldsymbol{x}_1=\sum_{i=1}^N\rho_{i\text{,}1}\boldsymbol{x}_{i\text{,}1}$\text{,}\\$\rho_{i\text{,}1}=(N+1)/[i(i+1)N]\text{,}\ \forall i\in\{1\text{,}2\text{,}\cdot\cdot\cdot\text{,}N\}.$
}
\StepTwo{
    \For{$t=2$ \KwTo $T$ }{
      Update SoC reference via kernel regression~\eqref{refupdate}. \;
      \For{$i=1$ \KwTo $N$ \Parallel}{
      Update virtual queue $Q_{i\text{,}t}$ and decisions $\boldsymbol{x}_{i\text{,}t}$ \; in parallel via~\eqref{eq:Q_update}-\eqref{eq:x_update}; Calculate surrogate\; loss in parallel via~\eqref{eq:loss}.  
      }
      Calculate expert weights and decisions via~\eqref{eq:final_decision}
    }
}
\end{algorithm}

\section{Numerical Case Study}\label{Case Study}


We demonstrate the effectiveness of the proposed dispatch framework on an isolated microgrid in Alaska. The diagram of the test system is illustrated in~\ref{fig:system}, comprising 100 kW wind, 100 kW solar, 150 kW load, 50 kW diesel generation, 50 kW/200 kWh BES, and 100 kW/1000 kg LDES. The initial SoC and efficiency of BES are set to 0.5 and 90\%, respectively. The initial SoC and final SoC target of LDES are set to 0.2 and 0.5, respectively. The nonconvex model of LDES is adopted from the semi-empirical model~\cite{sanchez2018semi}. The ground-truth data for renewable generation and load power from 1984 to 2024 are publicly available~\cite{Data}.

The optimization is performed hourly over an entire year and implemented in MATLAB with Gurobi 12.0 solver. The computational environment is an Intel Core i9-13900HX CPU @ 2.30 GHz with 16 GB RAM.

\begin{figure}[t] 
    \setlength{\abovecaptionskip}{-0.1cm}  
    \setlength{\belowcaptionskip}{-0.1cm}   
\centerline{\includegraphics[width=0.9\columnwidth]{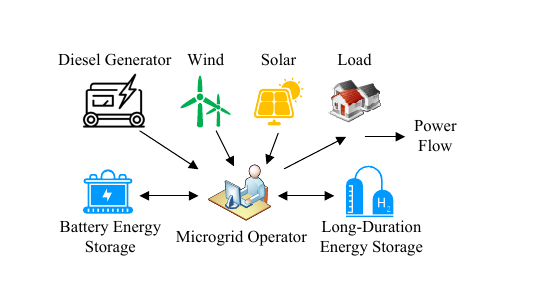}}
   \caption{Diagram of the test microgrid system.}    \label{fig:system}
\end{figure}

\subsection{Analysis on Nonconvex Model Approximation}

Table~\ref{table:nonconvex} compares the proposed convex hull approximation with the state-of-the-art piecewise linear approximation for the nonconvex LDES model. Both methods converge to near-optimality with increasing sample number (segments), but at the cost of increased computational time, indicating a trade-off between optimality and computational efficiency. Moreover, the constant efficiency model (sample number = 2) results in an optimality gap of 11.4\%, highlighting the advantage of the nonconvex and dynamic efficiency model, which accurately captures the electrochemical dynamics of LDES. 

Moreover, the piecewise linear method becomes intractable for offline trajectory generation when the sample number exceeds 60, whereas the convex hull method consistently maintains a low computational cost even with 100 samples. This is because the piecewise linear approximation introduces binary variables, causing the computational time to increase exponentially with the number of segments. Therefore, the piecewise linear method remains impractical for updating offline trajectories in response to operational mode changes or contingencies in microgrids. Furthermore, the binary variables introduced by this method exclude the application of most gradient-based online optimization techniques. Next, we employ the convex hull method with 100 samples for subsequent analysis.
\begin{table}[t]
\centering
\caption{Comparison of Operational Performance under Different Approximation Methods}
  \setlength{\tabcolsep}{0.3mm}{
    \begin{tabular}{ccccccc}
    \toprule
\multirow{2}[4]{*}{\makecell{Sample\\Number}} & \multicolumn{2}{c}{Annual Cost ($10^3$\$)} & \multicolumn{2}{c}{Annual Runtime (s)} & \multicolumn{2}{c}{Single-Period Runtime (s)}\\
\cmidrule{2-7}          & \makecell{Piecewise\\Linear} & \makecell{Convex\\Hull} & \makecell{Piecewise\\Linear} & \makecell{Convex\\Hull} & \makecell{Piecewise\\Linear} & \makecell{Convex\\Hull} \\
    \midrule
    2 & 39.47  & 39.47  & 19.03  & 8.33  & 0.06  & 0.01  \\
    10    & 35.56  & 35.55  & 181.09  & 11.65  & 0.07  & 0.01  \\
    20    & 35.45  & 35.45  & 268.18  & 12.65  & 0.06  & 0.01  \\
    40    & 35.44  & 35.44  & 3948.36  & 26.33  & 0.07  & 0.02  \\
    60    & \multirow{3}[1]{*}{——} & 35.44  & \multirow{3}[1]{*}{$>$2 h} & 29.45  & 0.07  & 0.03  \\
    80    &       & 35.44  &       & 24.70  & 0.07  & 0.05  \\
    100   &       & 35.44  &       & 24.98  & 0.07  & 0.05  \\
    \bottomrule
\end{tabular}\label{table:nonconvex}
}
\end{table}

\begin{figure}[ht!] 
    \setlength{\abovecaptionskip}{-0.1cm}  
    \setlength{\belowcaptionskip}{-0.1cm}   
\centerline{\includegraphics[width=0.9\columnwidth]{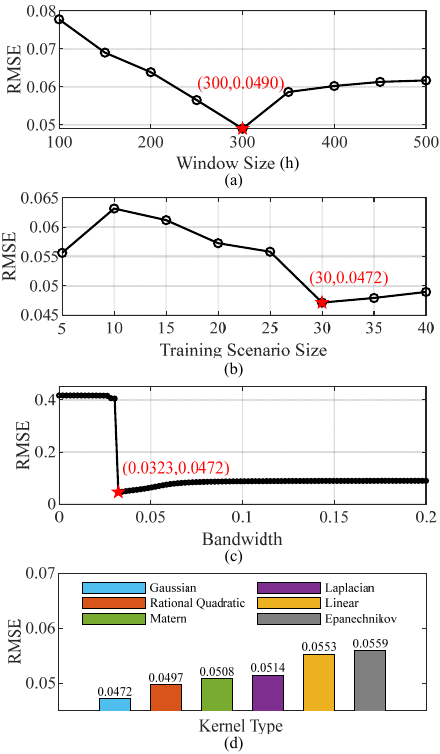}}
    \caption{Training process and testing performance of kernel regression.}   \label{fig:kernel}
\end{figure}

\subsection{Analysis on Online SoC Reference Learning Performance}

\subsubsection{Effectiveness of the proposed kernel regression} we first illustrate the training process of the proposed kernel regression method in Fig.\ref{fig:kernel}. In Fig.\ref{fig:kernel} (a), the RMSE initially decreases with an increasing window size and then rises, indicating an optimal window size at approximately 300 h, thereby empirically validating the theoretical analysis presented in~\eqref{eq:mse_kernel}. This optimal window is notably close to the duration of LDES (around 333.33 h). Moreover, Fig.\ref{fig:kernel} (b) indicates that the training performance generally improves with an increasing training scenarios. This aligns with part of the theorem and the theoretical bound in~\eqref{optimal_MSE}. Furthermore, as illustrated in Fig.\ref{fig:kernel} (c), once the optimal window size is determined, the corresponding optimal bandwidth can be obtained, consistent with the theoretical result provided in~\eqref{sigmaoptimal}. We also compare different kernel types, with optimally tuned hyperparameters in Fig.~\ref{fig:kernel} (d). The Gaussian kernel significantly outperforms the others due to its smoothness and infinite differentiability, enabling better learning of the underlying relationships.

\subsubsection{Comparative analysis on SoC reference learning} we benchmark the proposed kernel regression against several state-of-the-art long-term reference learning methods:

\textbf{(i) Kernel-1}: The proposed kernel regression method, which utilizes both historical SoC and netload for training.

\textbf{(ii) Kernel-2}: The kernel regression method proposed in~\cite{guo2023long,qi2025long}, which utilizes only historical netload for training.

\textbf{(iii) Average}: A rule-based method using the average of historical SoC trajectories.

\textbf{(iv) Long Short-Term Memory} and \textbf{(v) Convolutional Neural Network}: two neural network-based methods~\cite{kim2019predicting}.

We compare the SoC trajectories learned by the first three methods against the hindsight-optimal trajectory in Fig.~\ref{fig:reference} (a). The results of the last two methods, which keep the SoC around 0.45 with relatively poor performance, are excluded from the figure to avoid cluttering the comparison. The RMSE values of the five methods are 0.047, 0.070, 0.087, 0.128, and 0.123, respectively, with neural network-based methods performing the worst, highlighting their limitations in capturing long-term coupled dynamics. Compared to the average historical value, kernel regression dynamically updates the SoC reference by incorporating the most recent observations and adjusted weights, as shown in Fig.~\ref{fig:reference} (b), resulting in superior performance during the winter and spring seasons. Furthermore, compared to previous works on kernel regression, the proposed method further reduces the RMSE by 32.8\% through the incorporation of historical SoC values. This significantly reduces the variance by introducing new relevant factors, as future SoC dynamics depend on both future netload and past SoC.
\begin{figure}[t] 
    \setlength{\abovecaptionskip}{-0.1cm}  
    \setlength{\belowcaptionskip}{-0.1cm}   
\centerline{\includegraphics[width=0.9\columnwidth]{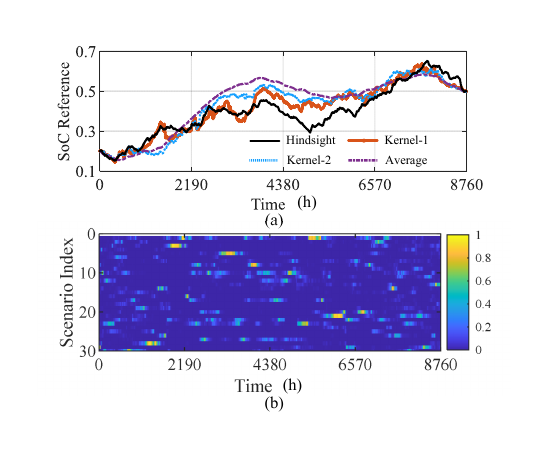}}
    \caption{LDES SoC trajectories learning results: (a) performance comparison across methods and (b) dynamically updated scenario weights.}    \label{fig:reference}
\end{figure}

\subsection{Analysis on Online Dispatch Performance}

\subsubsection{Effectiveness of the proposed OCO} Fig.~\ref{fig:penalty} (a) illustrates the performance of the proposed long-short-term coordinated dispatch method under different penalty coefficients for reference tracking. It is observed that the operational performance is sensitive to the penalty coefficients. As the penalty coefficients increase, the RMSE of long-term reference tracking decreases, hence the final SoC gradually approaches the contract requirement, aligning well with the theoretical results in~\eqref{penalty_error}. For smaller penalty coefficients (below 1500), increasing the coefficient greatly reduces annual costs and load loss by mitigating the myopia of SED. We achieve a minimal cost of \$39,573.96 (marked with stars), but the final SoC of 0.31 fails to meet the contract requirement. For larger penalty coefficients (above 10,000), increasing the coefficient reduces long-term contract violations but sacrifices annual costs due to the gap between the learned SoC reference and the hindsight-optimal SoC. We achieve minimal long-term contract violations (final SoC = 0.496, marked with squares), and this results in a significantly high cost of \$75,632.45. In practice, we select the trade-off decision (marked with cross marks) and account for the contract violation penalty of \$10/kg, resulting in a final SoC of 0.45 at a final cost of \$45351.00.

Compared to the hindsight-optimal solution, Fig.~\ref{fig:penalty} (b) shows the regret change rate for the proposed method. Initially, the regret change rate is relatively large, then gradually reduces before eventually approaching zero. This verifies the sublinear regret bound as stated in the theorem.
\begin{figure}[t] 
    \setlength{\abovecaptionskip}{-0.1cm}  
    \setlength{\belowcaptionskip}{-0.1cm}   
\centerline{\includegraphics[width=0.95\columnwidth]{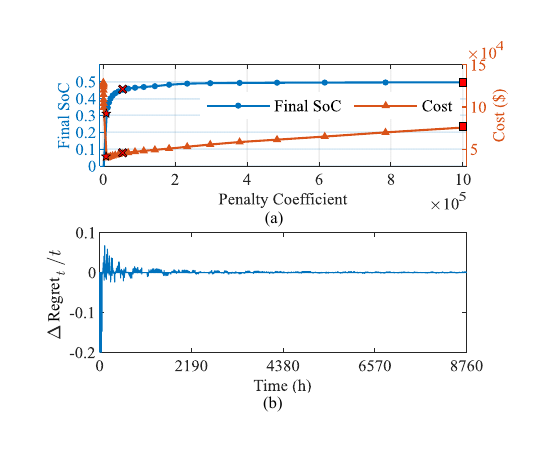}}
    \caption{Performance of the proposed method: (a) varying with penalty coefficients and (b) regret change rate.}    \label{fig:penalty}
\end{figure}

\subsubsection{Comparative analysis on online dispatch performance} we benchmark the proposed method against several state-of-the-art online optimization methods:

\textbf{(i) Perfect Foresight}: Under perfect knowledge of future uncertainties, we solve the OED using reformulation~\eqref{eq:reformulation}, serving as a hindsight-optimal baseline.

\textbf{(ii) OCO-1} The proposed long-short-term coordinated dispatch uses OCO with reference tracking for short-term control and Kernel-1 for long-term reference.

\textbf{(iii) MPC-1}: The long-short-term coordinated dispatch uses MPC with reference tracking~\cite{guo2023long} for short-term control and Kernel-1 for long-term reference.

\textbf{(iv) OCO-2}: The short-term dispatch based on SED formulation~\eqref{eq:SED} uses the proposed OCO without reference tracking.

\textbf{(v) MPC-2}: The short-term dispatch based on SED formulation~\eqref{eq:SED} uses the MPC without reference tracking.

Fig.~\ref{fig:opt} compares the operational performance of different methods, and Table~\ref{table:opt} summarizes their performance with various references. With the help of the long-term SoC reference, both MPC and OCO methods closely follow the hindsight-optimal SoC trajectory, resulting in relatively low cost and load loss. In contrast, short-term dispatch using either MPC or OCO generates myopic decisions, nearly depleting LDES during the winter without accounting for future opportunities and risks. This results in significant load loss (over 26 MWh) and regret (over \$140K) in subsequent seasons. Compared to MPC, the proposed OCO method leverages the most recent information without relying on future uncertainty, offering more reliable and robust solutions. Specifically, OCO achieves the lowest regret (\$9901.89), outperforming MPC with a cost reduction of 2.4\%. Furthermore, incorporating the reference has a significant impact on operational performance, enabling cost savings of approximately 73.4\% compared to operations without reference tracking. Compared with the average reference and the state-of-the-art reference, the proposed reference (kernel-1) further reduces costs by 2.8\% and 0.4\%, respectively. These results conclusively demonstrate the significant benefits of the proposed methods.
\begin{figure}[t] 
    \setlength{\abovecaptionskip}{-0.1cm}  
    \setlength{\belowcaptionskip}{-0.1cm}   
\centerline{\includegraphics[width=0.95\columnwidth]{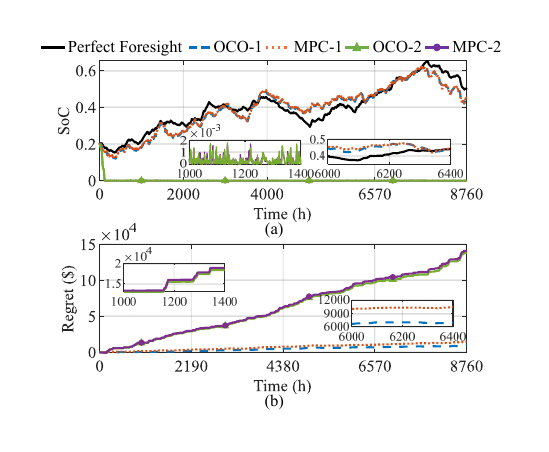}}
    \caption{Comparison of operational performance across different methods: (a) LDES SoC and (b) regret.}    \label{fig:opt}
\end{figure}
\begin{table}[t]
\centering
\caption{Comparison of Operational Performance with Differnt Methods and References}
  \setlength{\tabcolsep}{0.4mm}{
    \begin{tabular}{cccccc}
    \toprule
    \multirow{2}[4]{*}{Method} & \multirow{2}[4]{*}{Metric} & \multicolumn{4}{c}{Reference} \\
\cmidrule{3-6}          &       & Kernel-1 & Kernel-2 & Average & None \\
    \midrule
    \multirow{2}[1]{*}{Perfect Foresight} & Cost (\$) & \multicolumn{4}{c}{35449.11 } \\
          & Load Loss (kWh) & \multicolumn{4}{c}{0.00 } \\
    \multirow{2}[0]{*}{MPC} & Cost (\$) & 46419.87  & 48137.68  & 46508.56  & 176633.05  \\
          & Load Loss (kWh) & 0.00  & 0.00  & 0.00  & 26864.51  \\
    \multirow{2}[1]{*}{OCO} & Cost (\$) & \textbf{45351.00}  & 46271.28  & 45632.29  & 173741.53  \\
          & Load Loss (kWh) & \textbf{0.00}  & 0.00  & 0.00  & 26365.07  \\
    \bottomrule
\end{tabular}\label{table:opt}
}
\end{table}

\subsection{Sensitivity Analysis} 

\subsubsection{Results Sensitive to LDES Duration}

We evaluate the online learning and dispatch performance of the proposed method with varying LDES durations in Table~\ref{table:duration}. Increasing duration significantly reduces the RMSE, and the optimal window sizes closely match the LDES durations. This occurs because increasing the duration significantly reduces the SoC variations, from 0.74 at 500 kg to 0.382 at 2000 kg. This substantially lowers the Lipschitz constant, thereby reducing the RMSE of kernel regression, as theoretically established in~\eqref{eq:mse_kernel}. This indicates that kernel regression is better suited for long-duration storage but may have limitations when applied to short-duration storage. Moreover, for longer-duration LDES, the improved reference helps reduce the optimality gap of the proposed method from 34.83\% at 500 kg to 19.89\% at 2000 kg.
\begin{table}[t]
\centering
\caption{Comparison of Online Learning and Dispatch Performance under Different LDES Durations}
  \setlength{\tabcolsep}{0.35mm}{
    \begin{tabular}{cccccc}
    \toprule
    \makecell{LDES\\Capacity (kg)} & RMSE & \makecell{Optimal Window\\Size (h)} & \makecell{Optimal\\ Bandwidth} & SoC Range & \makecell{Optimality\\Gap (\%)} \\
    \midrule
    500   & 0.085  & 200   & 0.038  & 0.740 & 34.83 \\
    1000  & 0.047  & 300   & 0.032  & 0.499 & 27.93  \\
    1500  & 0.036  & 550   & 0.030  & 0.422  & 25.06\\
    2000  & 0.030  & 600   & 0.030  & 0.382 & 19.89 \\
    \bottomrule
\end{tabular}\label{table:duration}
}
\end{table}

\subsubsection{Results Sensitive to Forecast Error} As the Mean Absolute Percentage Error (MAPE) of forecast error increases, OCO outperforms MPC more significantly, with improvements of 2.30\% at 10\% MAPE, 10.62\% at 20\% MAPE, and 17.88\% at 30\% MAPE. In contrast, OCO remains stable regardless of forecast errors, as it only leverages previous information. Additionally, incorporating forecasts can further enhance the performance of OCO, as demonstrated in~\cite{chen2015online}.

\subsubsection{Results Sensitive to System Faults} 

Unexpected faults are inevitable in practical microgrid operations. We test the resilience of the proposed method in handling faults, as shown in Fig.~\ref{fig:adapt}. Upon the wind fault, offline SoC trajectories are regenerated based on current system states within minutes via parallel computing, providing updated references for online tracking. The wind fault causes an energy shortage, shifting the reference from charging trend (red) to discharging trend (yellow). After recovery, the reference guides LDES back to the charging trend until the DG fault occurs, shifting the reference to green lines. This adaptive reference closely aligns hindsight-optimal SoC trajectory, enabling near-optimal performance. The results show the resilience of the proposed method, offering significant advantages over learning-based methods, which require extensive retraining when system states change.
\begin{figure}[t] 
    \setlength{\abovecaptionskip}{-0.1cm}  
    \setlength{\belowcaptionskip}{-0.1cm}   
\centerline{\includegraphics[width=0.95\columnwidth]{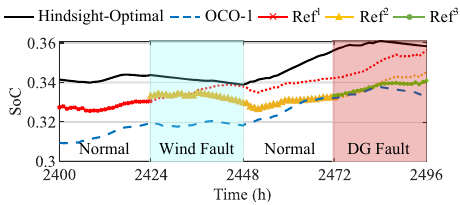}}
    \caption{Performance of the proposed method with adaptive references under unexpected system faults.}    \label{fig:adapt}
\end{figure}

\section{Conclusion and Discussion}\label{Conclusion}
This paper proposes a novel non-anticipatory long-short-term coordinated dispatch framework for isolated microgrid operation with hybrid BES and LDES. The proposed convex hull approximation effectively addresses the nonconvexity in the LDES model, ensuring computational tractability, scalability, and accuracy. To address temporal coupling and ensure non-anticipativity, we generate hindsight-optimal SoC trajectories of LDES and netloads for offline training. In the online stage, we employ kernel regression to dynamically update SoC reference and propose an adaptive OCO algorithm with reference and expert tracking. We prove that both long-term and short-term policies achieve sublinear dynamic regret bounds over time, and the bounds decrease with more kernel regression scenarios, larger reference-tracking penalty, and more convex-hull segments.  Simulation results on an isolated microgrid in Alaska demonstrate that the proposed method significantly outperforms state-of-the-art methods, achieving substantial system cost reductions and reliable power supply through reference tracking and the ``0-look-ahead'' OCO algorithm. The performance further improves with longer LDES durations, demonstrating resilience to poor forecasts and unexpected system faults.


\bibliographystyle{IEEEtran}
\bibliography{IEEEabrv,OCO}

\appendix

\subsection{Proof of Equivalence between Long-Term Trajectory Tracking and Explicit Time-Coupling Constraints} \label{appendix1}

On one hand, since the reference $\hat{h}_{t}$ satisfies the time-coupling constraints~\eqref{E-power-H-1}, enforcing a quadratic tracking term with a sufficiently large penalty $\theta$ implicitly ensures compliance with these constraints. On the other hand, through partial Lagrangian relaxation of constraint~\eqref{E-balance-H}, an additional penalty term emerges in the objective. In online dispatch, we can substitute SoC reference into the future decisions, which yields:
\begin{align}
\varphi(h_T-{H}_{T})^2&=\varphi(h_{t-1}+\sum_{\tau=t}^{T}\sum_{m\in\mathcal{M}}(\hat{\lambda}_{m\text{,}\tau}^{-}H_{m}^{-}-\hat{\lambda}_{m\text{,}\tau}^{+}H_{m}^{+})-{H}_{T})^2\notag\\
&=\varphi(h_t-\hat{h}_t)^2
\end{align}
where $\varphi$ is the Lagrange multiplier. Hence, we finish the proof.

\subsection{Proof of Theorem 1} \label{appendix2}

The dynamic regret is defined as the cumulative objective difference between the decisions obtained from the proposed method and the oracle optimal decisions, as illustrated in~\eqref{regret_def}. To rigorously analyze the regret, we decompose it into four components in~\eqref{regret_decomposition}: (1) convex hull approximation error, (2) OCO regret, (3) tracking penalty error, and (4) optimality gap.
\begin{subequations}
\begin{align}
&   \text{Regret}= \sum\nolimits_{t=1}^T[\hat{f}_t(x_t\text{,}y_t)-\hat{f}_t(x_t^{\star}\text{,}y_t^{\star})]\label{regret_def}  \\
& \text{Regret}=\underbrace{\sum\nolimits_{t=1}^T[\hat{f}_t(x_t\text{,}y_t)-\tilde{f}_t(x_t\text{,}y_t)+\tilde{f}_t(x_t^{\star}\text{,}y_t^{\star})-\hat{f}_t(x_t^{\star}\text{,}y_t^{\star})]}_{\text{Convex Hull Approximation Error}}\notag\\
&+\underbrace{\sum\nolimits_{t=1}^T[f(x_t\text{,}y_t)-f(x_t^\dagger\text{,}y_t^\dagger)]}_{\text{OCO Regret}}+\underbrace{\sum\nolimits_{t=1}^T\theta[\|h_t^\dagger-\hat{h}_t\|^2-\|h_t-\hat{h}_t\|^2]}_{\text{Tracking Penalty Error}}\notag\\
&+\underbrace{\sum\nolimits_{t=1}^T[\tilde{f}_t(x_t^\dagger\text{,}y_t^\dagger)-\tilde{f}_t(x_t^{\star}\text{,}y_t^{\star})]}_{\text{Optimality Gap}}\label{regret_decomposition}
\end{align}    
\end{subequations}
where $\hat{f}_t$ and $\tilde{f}_t$ denote the objective function of OED~\eqref{eq:OED} and convex reformulation~\eqref{eq:reformulation}, respectively. $(x_t^\star,y_t^\star)$ and $(x_t^\dagger,y_t^\dagger)$ denote the optimal solution to OED~\eqref{eq:OED} and TED~\eqref{eq:TED}, respectively. $h_t^\dagger$ denotes the optimal LDES SoC of TED.

\textbf{(1) Convex Hull Approximation Error.} Given a twice differentiable nonconvex power-hydrogen function with second derivative bounded by Z, from the convergence property of the Hausdorff distance between the nonconvex set and its convex hull approximation~\cite{chen2018convex}, we have:
\begin{align}
\text{Regret}^{\text{Approx}}\leq \dfrac{2c^\text{h}Z(\overline{P}-\underline{P})^2T}{(M-1)^2}=\mathcal{O}(\dfrac{T}{M^2}) 
\end{align}

\textbf{(2) OCO Regret.} From $f_{t}$ is convex and~\eqref{condition}, we have:
\begin{align}\label{reg1}
&f_t(\boldsymbol{x}_{i,t}) - f_t(\boldsymbol{x}_t^{\dagger}) \leq \left\langle \partial f_t(\boldsymbol{x}_{i,t}),\ \boldsymbol{x}_{i,t} - \boldsymbol{x}_t^{\dagger} \right\rangle \\
&\leq J \parallel \boldsymbol{x}_{i,t} - \boldsymbol{x}_{i,t+1} \parallel + \left\langle \partial f_t(\boldsymbol{x}_{i,t}),\ \boldsymbol{x}_{i,t+1} - \boldsymbol{x}_t^{\dagger} \right\rangle\notag \\
&\leq \frac{J^{2} \alpha_{i,t}}{2} + \frac{1}{2 \alpha_{i,t}} \parallel \boldsymbol{x}_{i,t} - \boldsymbol{x}_{i,t+1} \parallel^{2} + \left\langle \partial f_t(\boldsymbol{x}_{i,t}),\ \boldsymbol{x}_{i,t+1} - \boldsymbol{x}_t^{\dagger} \right\rangle\notag
\end{align}

For the rightmost term of~\eqref{reg1}, we have:
\begin{align}\label{reg2}
&\left\langle\partial f_t(\boldsymbol{x}_{i,t}),\ \boldsymbol{x}_{i,t+1}-\boldsymbol{x}_t^{\dagger} \right\rangle 
\\&=\left\langle\beta_{i,t+1}(\partial[g_t(\boldsymbol{x}_{i,t+1})]_{+})^{T}Q_{i,t},\ \boldsymbol{x}_t^{\dagger} -\boldsymbol{x}_{i,t+1}\right\rangle \notag
\\&+\left\langle\partial f_t(\boldsymbol{x}_{i,t})+\beta_{i,t+1}(\partial[g_t(\boldsymbol{x}_{i,t+1})]_{+})^{T}Q_{i,t},\ \boldsymbol{x}_{i,t+1}-\boldsymbol{x}_t^{\dagger} \right\rangle\notag
\end{align}

Since $g_{t}$ is a convex function, it is trivial to show that $[g_t]_+$ is also convex; hence the first term of~\eqref{reg2} can be relaxed:
\begin{align}\label{reg3}
&\left\langle\beta_{t+1}(\partial[g_t(\boldsymbol{x}_{i,t+1})]_{+})^{T}Q_{i,t},\ \boldsymbol{x}_t^{\dagger} -\boldsymbol{x}_{i,t+1}\right\rangle\\
&\leq\beta_{t+1}\left\langle Q_{i,t},\ [g_t(\boldsymbol{x}_t^{\dagger} )]_+\right\rangle-\beta_{i,t+1}\left\langle Q_i(t),\ [g_t(\boldsymbol{x}_{i,t+1})]_+\right\rangle\notag
\end{align}

From Lemma 1 in~\cite{yi2020distributed}, we have:
\begin{align}\label{reg4}
&\left\langle\partial f_t(\boldsymbol{x}_{i,t})+\beta_{i,t+1}(\partial[g_t(\boldsymbol{x}_{i,t+1})]_{+})^{T}Q_{i,t},\ \boldsymbol{x}_{i,t+1}-\boldsymbol{x}_t^{\dagger} \right\rangle\\
&\leq\frac1{\alpha_{i,t}}(\parallel \boldsymbol{x}_t^{\dagger} -\boldsymbol{x}_{i,t}\parallel^2-\parallel \boldsymbol{x}_t^{\dagger} -\boldsymbol{x}_{i,t+1}\parallel^2-\parallel \boldsymbol{x}_{i,t+1}-\boldsymbol{x}_{i,t}\parallel^2)\notag
\end{align}

Combining~\eqref{eq:final_decision},~\eqref{reg1}-\eqref{reg4}, we have:
\begin{align}\label{reg5}
&\ell_t(\boldsymbol{x}_{i,t}) - \ell_t(\boldsymbol{x}_t^{\dagger} ) \leq{J^2\alpha_{i,t}}/2+\beta_{t+1}\left\langle Q_i(t),\ [g_t(\boldsymbol{x}_t^{\dagger} )]_+\right\rangle
\\&(\parallel \boldsymbol{x}_t^{\dagger} -\boldsymbol{x}_{i,t}\parallel^2-\parallel \boldsymbol{x}_t^{\dagger} -\boldsymbol{x}_{i,t+1}\parallel^2)/\alpha_{i,t}\notag
\end{align}

Since the second term of~\eqref{reg5} is non-negative, we have:
\begin{align}\label{reg6}
&\sum\nolimits_{t=1}^{T} (\ell_t(\boldsymbol{x}_{i,t}) - \ell_t(\boldsymbol{x}_t^{\dagger} ))\leq\sum\nolimits_{t=1}^T{J^2\alpha_{i,t}}/2\\
&+\sum\nolimits_{t=1}^T(\parallel \boldsymbol{x}_t^{\dagger} -\boldsymbol{x}_{i,t}\parallel^2-\parallel \boldsymbol{x}_t^{\dagger} -\boldsymbol{x}_{i,t+1}\parallel^2)/{\alpha_{i,t}}\notag
\end{align}

For the first term of~\eqref{reg6}, we have:
\begin{equation}\label{reg7}
\begin{aligned}
\sum_{t=1}^T\frac{J^2\alpha_{i,t}}2\leq\frac{2^{i-1}J^2}2\sum_{t=1}^T\frac1{t^{c}}\leq\frac{2^{i-1}J^2}{2(1-c)}T^{1-c}
\end{aligned}
\end{equation}

By leveraging~\eqref{condition} and parameter setting~\eqref{parameter}, we have:
\begin{align}\label{reg8}
&\sum_{t=1}^{T} \frac{t^c}{\alpha_0 2^{i-1}} \left( \|\boldsymbol{x}_t^{\dagger}  - \boldsymbol{x}_{i,t}\|^2 - \|\boldsymbol{x}_t^{\dagger}  - \boldsymbol{x}_{i,t+1}\|^2 \right) \\
&= \frac{1}{\alpha_0 2^{i-1}}\sum_{t=1}^{T} \big( t^c \|\boldsymbol{x}_t^{\dagger}  - \boldsymbol{x}_{i,t}\|^2 - (t + 1)^c \|\boldsymbol{x}_{t+1}^{\dagger} - \boldsymbol{x}_{i,t+1}\|^2 \notag\\
&+  (t + 1)^c \|\boldsymbol{x}_{t+1}^{\dagger} - \boldsymbol{x}_{i,t+1}\|^2 - t^c \|\boldsymbol{x}_t^{\dagger}  - \boldsymbol{x}_{i,t+1}\|^2\notag  \\
&+ t^c \|\boldsymbol{x}_t^{\dagger}  - \boldsymbol{x}_{i,t+1}\|^2 - t^c \|\boldsymbol{x}_t^{\dagger}  - \boldsymbol{x}_{i,t}\|^2 \big)\notag\\
&\le \frac{1}{\alpha_0 2^{i-1}} \|\boldsymbol{x}_{1}^{\dagger} - \boldsymbol{x}_{i,1}\|^2 + \frac{1}{\alpha_0 2^{i-1}} \sum_{t=1}^{T} \left( (t + 1)^c - t^c \right) (d(\mathcal{X}))^2 \notag\\
&+ \frac{2}{\alpha_0 2^{i-1}} \sum_{t=1}^{T} t^c d(\mathcal{X}) \|\boldsymbol{x}_{t+1}^{\dagger} - \boldsymbol{x}_t^{\dagger} \|\notag \\
&\le \frac{1}{\alpha_0 2^{i-1}} \left( 1 + (T + 1)^c - 1 \right) (d(\mathcal{X}))^2 + \frac{2T^c d(\mathcal{X}) P_x}{\alpha_0 2^{i-1}}\notag \\
&\le \frac{2}{\alpha_0 2^{i-1}} (d(\mathcal{X}))^2 T^c \left( 1 + {P_x}/{d(\mathcal{X})} \right)\notag
\end{align}
where $d(\mathcal{X})$ denotes the diameter of the feasible set.

Let $i_0=\left[\frac12\log_2(1+{P_x}/{d(\mathcal{X})})\right]+1\in[N]$, we have:
\begin{equation}\label{reg9}
2^{i_0-1}\leq\sqrt{1+{P_x}/{d(\mathcal{X})}}\leq2^{i_0}
\end{equation}

Combining~\eqref{reg7}-\eqref{reg9} yields:
\begin{align}\label{reg10}
&\sum_{t=1}^{T} (\ell_t(\boldsymbol{x}_{i_0, t}) - \ell_t(\boldsymbol{x}_t^{\dagger} )) \leq \frac{4}{\alpha_0} (d(\mathcal{X}))^2 T^c \left( 1 + \frac{P_x}{d(\mathcal{X})} \right)^{1-\kappa}\notag\\
&+ \frac{J^2 \alpha_0}{2 (1 - c)} T^{1-c} \left( 1 + \frac{P_x}{d(\mathcal{X})} \right)^{\kappa}
\end{align}
Applying Lemma 1 in reference~\cite{zhang2018adaptive} to~\eqref{eq:final_decision} yields:
\begin{equation}\label{reg11}
\begin{aligned}
\sum_{t=1}^T\ell_t(\boldsymbol{x}_{t})-\min_{i\in[N]}\{\sum_{t=1}^T\ell_t(\boldsymbol{x}_{i,t})+\frac1\gamma\ln\frac1{\rho_{i,1}}\}\leq\frac{\gamma(Jd(\mathcal{X}))^2T}2
\end{aligned}
\end{equation}
\begin{equation}\label{reg12}
\begin{aligned}
\sum_{t=1}^T(\ell_t(\boldsymbol{x}_{t})-\ell_t(\boldsymbol{x}_{i_0,t}))\leq\frac{\gamma_0 (J d(\mathcal{X}))^2 T^{1 - c}}{2} + \frac{1}{\gamma_0} T^c \ln \frac{1}{\rho_{i_0, 1}}
\end{aligned}
\end{equation}
From $\rho_{i,1}=(N+1)/[i(i+1)N]$, we have:
\begin{equation}\label{reg13}
\begin{aligned}
\ln\frac1{\rho_{i_0,1}}\leq\ln(i_0(i_0+1))\leq2\ln(i_0+1)\leq2\ln(\left\lfloor\kappa\log_2(1+\frac{P_x}{d(\mathcal{X})})\right\rfloor)
\end{aligned}
\end{equation}
From~\eqref{eq:final_decision} and that $f_{t}$ is convex, we have
\begin{equation}\label{reg14}
\begin{aligned}
f_t(\boldsymbol{x}_{t})-f_t(\boldsymbol{x}_t^{\dagger} )\leq \ell_t(\boldsymbol{x}_{t})-\ell_t(\boldsymbol{x}_t^{\dagger} )
\end{aligned}
\end{equation}
Combining~\eqref{reg10}-\eqref{reg14} yields:
\begin{equation}\label{reg15}
\begin{aligned}
&\text{Regret}\leq \frac{4}{\alpha_0} (d(\mathcal{X}))^2 T^c \left( 1 + \frac{P_x}{d(\mathcal{X})} \right)^{1 - \kappa}+\frac{\gamma_0 (J d(\mathcal{X}))^2 T^{1 - c}}{2}\\
 &+\frac{J^2 \alpha_0}{2(1 - c)} T^{1 - c} \left( 1 + \frac{P_x}{d(\mathcal{X})} \right)^{\kappa}+\frac{2}{\gamma_0} T^c \ln ([\kappa \log_2 \left( 1 + \frac{P_x}{d(\mathcal{X})} \right)])\\
\end{aligned}
\end{equation}

Hence, we have the regret bound of OCO in~\eqref{OCObound}.
\begin{align}\label{OCObound}
\text{Regret}^{\text{OCO}}=\mathcal{O}(T^c (1 + P_x)^{1 - \kappa} + T^{1 - c} (1 + P_x)^{\kappa})    
\end{align}

\textbf{(3) Tracking Penalty Error.} From the KKT condition of TED, we have~\eqref{KKT}, where $\varphi^\dagger$ denotes the Lagrange multiplier of TED. Since the second term of tracking penalty error is non-positive, we have~\eqref{penalty_error}. 
\begin{subequations}
\begin{align}
&\nabla_{h_t}\tilde{f}_t(x_t\text{,}y_t)+2\theta(h_t^\dagger-\hat{h}_t)+\varphi^\dagger=0\label{KKT}\\
&\text{Regret}^{\text{Penalty}}\leq\sum\nolimits_{t=1}^T\theta\|h_t^\dagger-\hat{h}_t\|^2=\mathcal{O}(\theta\dfrac{T}{{\theta^2}})=\mathcal{O}(\dfrac{T}{{\theta}})\label{penalty_error}
\end{align}    
\end{subequations}

\textbf{(4) Optimality Gap.} Due to the Lipschitz continuity of $\tilde{f}_t$, we have~\eqref{L_f}. From the power balance constraints~\eqref{powerbalance-1} and SoC dynamics~\eqref{E-power-H-1}, the relationship between \(h_t\) and \((x_t\text{,} y_t)\) is affine. Thus, we have~\eqref{L_H}. Then, by applying the triangle inequality, we have~\eqref{triangle_inequality}.
\begin{subequations}
\begin{align}
&\tilde{f}_t(x_t^\dagger\text{,}y_t^\dagger)-\tilde{f}_t(x_t^{\star}\text{,}y_t^{\star})\leq\Lambda_f\|(x_t^\dagger\text{,}y_t^\dagger)-(x_t^{\star}\text{,}y_t^{\star})\|\label{L_f}\\
& \|(x_t^\dagger,y_t^\dagger)-(x_t^\star,y_t^\star)\| \leq \Lambda_h \|h_t^\dagger - {h}_t^\star\|\label{L_H}\\
&\|h_t^\dagger - {h}_t^\star\|\leq\|h_t^\dagger - \hat{h}_t\|+\|\hat{h}_t - \tilde{h}_t\|+\|\tilde{h}_t - {h}_t^\star\| \label{triangle_inequality}
\end{align}
\end{subequations}
where $\Lambda_f$ and $\Lambda_h$ denote the Lipschitz constants of $\tilde{f}_t$ and the affine mapping from decision variables to LDES SoC.

From~\eqref{eq:mse_kernel}, we have the best bandwidth: $\sigma^\star=\mathcal{O}({S}^{-1/(4+\iota)})$. By substituting it into~\eqref{eq:mse_kernel}, we can obtain the optimal MSE in~\eqref{optimal_MSE}. Thus, we have the tracking optimality gap in~\eqref{optimality_gap}.
\begin{subequations}
\begin{align}
& \text{MSE}^\star=\mathbb{E}\|h_t^\dagger-\hat{h}_t\|^2=\mathcal{O}({S}^{-4/(4+\iota)})\label{optimal_MSE} \\
&\text{Regret}^{\text{Opt}}=\mathcal{O}({T}/{{S}^{2/(4+\iota)}})+\mathcal{O}({T}/{\theta})+\mathcal{O}({T}/{M^2}) \label{optimality_gap}
\end{align}    
\end{subequations}

By substituting the parameter settings in~\eqref{parameter} into the overall regret expression given by~\eqref{overall_regret}, we have finished the proof.
\begin{align}\label{overall_regret}
\text{Regret}&=\mathcal{O}(T^c (1 + P_x)^{1 - \kappa} + T^{1 - c} (1 + P_x)^{\kappa})\\&+\mathcal{O}({T}/{{S}^{2/(4+\iota)}})+\mathcal{O}({T}/{\theta})+\mathcal{O}({T}/{M^2})  \notag 
\end{align}

\end{document}